\documentclass[11pt]{amsart}

\usepackage[a4paper,hmargin=3.5cm,vmargin=4cm]{geometry}
\usepackage{amsfonts,amssymb,amscd,amstext,verbatim}

\usepackage{fancyhdr}
\usepackage[utf8]{inputenc}
\usepackage{hyperref}
\usepackage{verbatim}
\input xy
\xyoption{all}

\usepackage{times}

\usepackage{enumerate}
\usepackage{titlesec}
\usepackage{mathrsfs}

\titleformat{\section}
{\filcenter\bfseries\large} {\thesection{.}}{0.2cm}{}
\titleformat{\subsection}[runin]
{\bfseries} {\thesubsection{.}}{0.15cm}{}[.]
\titleformat{\subsubsection}[runin]
{\em}{\thesubsubsection{.}}{0.15cm}{}[.]

\usepackage[up,bf]{caption}

\newtheorem{theorem}{Theorem}[section]
\newtheorem{proposition}[theorem]{Proposition}

\newtheorem{lemma}[theorem]{Lemma}

\theoremstyle{definition}
\newtheorem{definition}[theorem]{Definition}
\newtheorem{remark}[theorem]{Remark}

\newtheorem{problem}[theorem]{Problem}

\numberwithin{equation}{section}
\numberwithin{figure}{section}

\newcommand\Ascr{\mathscr{A}}

\newcommand\Cscr{\mathscr{C}}

\newcommand\Oscr{\mathscr{O}}

\newcommand\C{\mathbb{C}}

\newcommand\CP{\mathbb{CP}}

\newcommand\R{\mathbb{R}}

\newcommand\igot{\mathfrak{i}}

\renewcommand\igot{\mathfrak{i}}

\renewcommand\imath{\igot}

\newcommand\hra{\hookrightarrow}

\newcommand\wt{\widetilde}
\newcommand\wh{\widehat}

\numberwithin{equation}{section}

\begin{document}

\fancyhead[LO]{Stein neighbourhoods of bordered complex curves}
\fancyhead[RE]{F.\ Forstneri{\v c}} 
\fancyhead[RO,LE]{\thepage}

\thispagestyle{empty}

\vspace*{1cm}
\begin{center}
{\bf\LARGE Stein neighbourhoods of bordered complex curves \\ 
attached to holomorphically convex sets}

\vspace*{0.5cm}

{\large\bf  Franc Forstneri{\v c}} 
\end{center}

\vspace*{1cm}

\begin{quote}
{\small
\noindent {\bf Abstract}\hspace*{0.1cm}
In this paper, we construct open Stein neighbourhoods of compact sets of the form 
$A\cup K$ in a complex space, where $K$ is a compact holomorphically convex set,
$A$ is a compact complex curve with boundary $bA$ of class $\Cscr^2$ which may 
intersect $K$, and the set $A\cap K$ is $\Oscr(A)$-convex. 
\vspace*{0.2cm}

\noindent{\bf Keywords}\hspace*{0.1cm} Stein manifold, complex curve, holomorphically convex set

\vspace*{0.1cm}

\noindent{\bf MSC (2020):}\hspace*{0.1cm}} 
32E10; 32E20; 32U05   

\vspace*{0.1cm}
\noindent{\bf Date: \rm April 20, 2021. Final version: Oct. 6, 2022}

%
%
%
%
\end{quote}

%
%
%
%
\section{Introduction}\label{sec:intro}

An important problem in complex analysis is to understand which sets in a complex 
manifold or a complex space admit a basis of open Stein neighbourhoods.
This problem is of both theoretical and practical importance 
due to the abundance of analytic techniques available on Stein manifolds.
A seminal result of Siu \cite{Siu1976} from 1976 is that a locally closed Stein subspace $A$ of an
arbitrary complex space $X$ admits a basis of open Stein neighbourhoods in $X$. 
(See also Suzuki \cite[Lemma 3, p.\ 59]{Suzuki1978} for a special case, and 
Col{\c{t}}oiu \cite{Coltoiu1990} and Demailly \cite{Demailly1990} for simpler proofs and 
generalizations to $q$-convex subspaces.) Another exposition can be found in 
\cite[Sections 3.1--3.2]{Forstneric2017E} where it was shown in addition that if 
$A\subset X$ are as above and $K$ is a {\em compact holomorphically convex set} in $X$ 
(i.e., $K$ is $\Oscr(\Omega)$-convex in an open Stein neighbourhood $\Omega\subset X$ of $K$; 
see Definition \ref{def:holoconvex}) such that $A\cap K$ is a compact 
$\Oscr(A)$-convex set (see Sect.\ \ref{sec:prelim}), then $A\cup K$ admits a basis of open 
Stein neighbourhoods; see \cite[Theorem 3.2.1]{Forstneric2017E} or \cite[Theorem 1.2]{Forstneric2005AIF}. 

The problem of finding Stein neighbourhoods of Stein subvarieties with boundaries 
is more subtle. In this paper we prove the following result in this direction.

%
%
\begin{theorem} \label{th:AK}
Assume that $X$ is a complex space and $A$ is a compact complex curve in $X$ 
with embedded $\Cscr^2$ boundary and having no irreducible components without boundary.
If $K$ is a compact holomorphically convex set in $X$ such that $A\cap K$ is $\Oscr(A)$-convex, 
then $A\cup K$ has a basis of open Stein neighbourhoods (i.e., it is a Stein compact).
\end{theorem} 

Under the additional assumption that $bA\cap K=\varnothing$, this was proved in my 
joint work with Drinovec Drnov\v sek in 2007; see \cite[Theorem 2.1]{DrinovecForstneric2007DMJ}. 
(The special case with $K=\varnothing$ and $X=\CP^n$ is due to Mihalache \cite{Mihalache1996} in 1996.) 
Here, this condition is removed by an entirely new proof. 
Simple examples show that $\Oscr(A)$-convexity of $A\cap K$ is a necessary condition
in Theorem \ref{th:AK}; see \cite[Remark 3.2.2]{Forstneric2017E}. Hence, Theorem \ref{th:AK} is optimal, 
except perhaps in terms of the boundary regularity of the complex curve $A$. 
Piecewise $\Cscr^2$ boundary is fine since the corner points can be added 
to $K$, and $\Cscr^1$ boundary would be ideal. 

We point out that Theorem \ref{th:AK} fails in general if $A$ is a subvariety  
of higher dimension $\dim A>1$ whose boundary intersects $K$. The reason is that 
a nontrivial envelope of holomorphy of $A\cup K$ may appear along $bA\cap bK$, at least if $K$ 
has nonempty interior. For example, let $A$ be a closed ball in the complex $2$-plane 
$\Sigma=\C^2\times \{0\}\subset \C^3$, and let $K$ be a closed ball in $\C^3$ centred 
at $0$. If $A$ is not contained in $K$ but $bA$ intersects 
the interior of $K$, then $(A\cup K)\cap\Sigma$ is a union of two compact strongly 
pseudoconvex domains with a nontrivial envelope of holomorphy 
along $bA\cap bK$, and hence $A\cup K$ does not have a basis of Stein neighbourhoods. 
On the other hand, it was shown by Star\v ci\v c \cite[Theorem 1]{Starcic2008} that the 
analogue of Theorem \ref{th:AK} holds if $A$ is a compact complex subvariety 
with Stein interior whose embedded strongly pseudoconvex boundary $bA$ of class $\Cscr^2$ 
does not intersect $K$. 

My initial attempt to prove Theorem \ref{th:AK} by adapting the techniques used in 
\cite[proof of Theorem 2.1]{DrinovecForstneric2007DMJ} was not successful,
so I developed a new approach. An important ingredient in the present proof 
is the seminal theorem of Stolzenberg \cite{Stolzenberg1966AM} from 1966,  
concerning the polynomial hull of the union of a compact polynomially convex set and 
a finite number of $\Cscr^1$ curves in a Euclidean space $\C^N$.
Stolzenberg's result easily adapts to Stein spaces via the embedding theorem;
see Theorem \ref{th:Stolzenberg}. We first show that every set of the form $A\cup K$ as
in Theorem \ref{th:AK}, which we call {\em admisible}, 
can be slightly enlarged to a more regular {\em special admissible set}; 
see Lemma \ref{lem:enlarging}. By \cite[Lemma 2.4]{DrinovecForstneric2007DMJ} 
there exists a strongly plurisubharmonic function in a neighbourhood of $A\cup K$ in $X$.
The main work is to show that any special admissible set is holomorphically convex, 
hence a Stein compact (see Theorem \ref{th:AKS}); this proves Theorem \ref{th:AK}. 
The proof of Theorem \ref{th:AKS}, given in Section \ref{sec:proof}, is based on a 
couple of lemmas, the second of which amounts to welding pairs of 
holomorphically convex special admissible sets under a suitable geometric condition on their intersection.
(This intersection should be a union of pairwise disjoint complex discs and 
should not contain any complex annuli.) Theorem \ref{th:AKS} is obtained by a finite inductive application 
of this welding technique. A version of this technique was previously used by Poletsky \cite{Poletsky2013} 
to construct Stein neighbourhoods of certain graphs.

In this regard I wish to mention the following problem. 

\begin{problem}\label{prob:Euclidean}
Assume that $X$ is a complex manifold of dimension $n\ge 3$ and $K$ is a compact set in $X$
which admits an open neighbourhood $\Omega$ and a biholomorphic map 
$\Phi:\Omega\to \Phi(\Omega)\subset \C^n$ such that  $\Phi(K)\subset \C^n$ 
is polynomially convex. Let $A$ be a smooth complex curve as in Theorem \ref{th:AK}
such that $A\cap K$ is $\Oscr(A)$-convex. 
Does $A\cup K$ admit a Euclidean neighbourhood in $X$?
\end{problem}

It has recently been shown in \cite{Forstneric2022JMAA} that if $A$ is a locally closed 
smooth complex curve without boundary in a complex manifold $X$ of dimension $\ge 3$ 
such that $A\cap K$ is compact and $\Oscr(A)$-convex,
then, with $K$ as above, the union $A\cup K$ admits a basis of Euclidean Stein neighbourhoods. 
This is a special case of \cite[Theorem 1.1]{Forstneric2022JMAA}. 
However, it is often desirable to have a holomorphic coordinate system around the entire
object under consideration, including its boundary. 
At this time I am unable to answer Problem \ref{prob:Euclidean} due to technical problems in  
adapting the methods from \cite{Forstneric2022JMAA}. The only nontrivial problem seems to be
that the gluing lemma in the cited paper  (see \cite[Theorem 3.7]{Forstneric2022JMAA}) 
does not apply when the curve $A$ is compact with smooth but non-analytic boundary. 
If $A$ is analytic past the boundary along a suitable region in $A$ along which we wish to glue,
then \cite[Theorem 3.7]{Forstneric2022JMAA} applies and we obtain a Euclidean neighbourhood 
of $A\cup K$. (See \cite[Theorem 1.1]{Forstneric2022JMAA}. The assumption $n=\dim X\ge 3$
is needed to ensure that the same curve $A$ also embeds holomorphically into $\C^n$.)
We leave a formal statement and proof to an eventual application.

%
%
\section{Preliminaries}\label{sec:prelim}

In this section we recall some notions and results from the theory of 
complex spaces and holomorphic convexity, and we prepare the auxiliary results which will be used.
For the general theory of complex spaces, see \cite{GrauertRemmert1958,GrauertPeternellRemmert1994};
for Stein spaces, see \cite{GrauertRemmert1979,Hormander1990}.

Let $X$ be a complex space. We denote by $\Oscr(X)$ the algebra of holomorphic functions on $X$.
By definition, every point $p\in X$ admits an open neighbourhood 
$U\subset X$ and a holomorphic embedding $\phi : U\hra U'\subset \C^N$ onto a closed complex
subvariety $\phi(U)$ of a domain $U'$ in a complex Euclidean space $\C^N$.
By shrinking $U$ around $p$ and taking $N$ the smallest possible (this number is called
the local embedding dimension of $X$ at $p$), any two such local holomorphic embeddings
are related by a biholomorphism between a pair of their open neighbourhood in $\C^N$. 
This makes it possible to introduce differential calculus and the notion of smooth  
functions and related objects on complex spaces. A function $f$ on $X$ 
is said to be of class $\Cscr^r$ if it is locally near any point $p\in X$ of the form $f=g\circ \phi$, where 
$\phi:U\hra U'\subset\C^N$ is a local holomorphic embedding as above and $g$ is a $\Cscr^r$ function 
on the domain $U'$. The function $f$ is {\em (strongly) plurisubharmonic} if the local extension $g$ as above
can be chosen (strongly) plurisubharmonic on the ambient domain.  
A locally closed subset $M$ of a complex space $X$ is said to be
a {\em submanifold of class $\Cscr^r$} if every point $p\in M$ admits an open neighbourhood 
$U\subset X$ and a holomorphic embedding $\phi : U\hra U'\subset \C^N$ such that 
$\phi(M\cap U)$ is a locally closed submanifold of class $\Cscr^r$ in $\C^N$. 

%
%
Given a compact set $K$ in a complex space $X$, its {\em $\Oscr(X)$-convex hull} is the set
\[
	\wh K_{\Oscr(X)} = \bigl\{x\in X: |f(x)| \le \sup_K|f|\ \ \text{for all}\ f\in\Oscr(X)\bigr\}. 
\]
The set $K$ is called {\em $\Oscr(X)$-convex} if $K=\wh K_{\Oscr(X)}$.
If $X$ is a {\em Stein space} then the hull $\wh K_{\Oscr(X)}$ of any compact set $K$ 
is compact. In fact, this is one of two axioms characterizing Stein spaces, the other one 
being separation of points by holomorphic functions. 

According to Narasimhan \cite{Narasimhan1961} and Forn\ae ss and Narasimhan \cite{FornaessNarasimhan1980},
a complex space is Stein if and only if it admits a strongly plurisubharmonic exhaustion function. 
For manifolds, this is due to Grauert \cite[Theorem 2]{Grauert1958AM}; see also 
H\"ormander \cite[Theorem 5.2.10]{Hormander1990}. 
On a Stein space, the holomorphic hull $\wh K_{\Oscr(X)}$ equals the hull of $K$ 
with respect to the family of all (continuous or smooth) plurisubharmonic functions on $X$. 
By using the easy part of this result that the $\Oscr(X)$-convex hull contains the plurisubharmonic hull, 
we infer that if $K$ is a compact $\Oscr(X)$-convex set in a Stein space $X$ 
and $V\subset X$ is a neighbourhood of $K$, there are a smooth plurisubharmonic 
exhaustion function $\rho_0:X\to \R$ such that $\rho_0<0$ on $K$ and $\rho_0>0$ on 
$X\setminus V$ (see \cite[Theorem 5.1.6]{Hormander1990} and note that the proof given there also
applies to Stein spaces) and a smooth plurisubharmonic exhaustion function $\rho:X\to \R_+$  
such that $K=\{\rho=0\}$ and $\rho$ is strongly plurisubharmonic on $X\setminus K=\{\rho>0\}$ 
(see \cite[Lemma, p.\ 430]{Rosay2006} or \cite[Proposition 2.5.1]{Forstneric2017E}).
Furthermore, given an $\Oscr(X)$-convex set $K$ and a plurisubharmonic function 
$\rho$ in a neighbourhood $U$ of $K$, there is an exhausting plurisubharmonic function $\tilde \rho$ on $X$ 
which agrees with $\rho$ near $K$ and is strongly plurisubharmonic
near infinity. To see this, pick a neighborhoood $V\Subset U$ of $K$, a strongly
plurisubharmonic exhaustion function $\rho_0:X\to\R$ such that $\rho_0<0$ on $K$
and $\rho_0>0$ on $X\setminus V$, and a number $c_0>0$ such that $c_0\rho_0>\rho$ on $bV$ 
and $c_0\rho_0<\rho$ on $K$, and define $\tilde \rho=\max\{c_0\rho_0,\rho\}$
on $V$ and $\tilde \rho=c_0\rho$ on $X\setminus V$. By using the regularized maximum 
(see Demailly \cite{Demailly1990} or \cite[p.\ 69]{Forstneric2017E}) we can make $\tilde \rho$ smooth.

%
%

We shall need the following local version of the notion of holomorphic convexity.

\begin{definition}\label{def:holoconvex} 
A compact set $K$ in a complex space $X$ is {\em holomorphically convex} if there is an open 
Stein neighbourhood $\Omega\subset X$ of $K$ such that $K=\wh K_{\Oscr(\Omega)}$ is 
$\Oscr(\Omega)$-convex.
\end{definition}

Clearly, every compact holomorphically convex  set is a {\em Stein compact}, i.e., it has 
a basis of open Stein neighbourhoods. Although the converse is not true in general, 
a Stein compact $K$ can be approximated from the outside by compact holomorphically 
convex sets obtained by taking the hulls of $K$ in open Stein neighbourhoods of $K$.

We shall use the following characterization of holomorphically convex sets.
(This is a folklore result; a specific reference is \cite{Rosay2006}, but likely not the first one.)

%
%
\begin{proposition}\label{prop:hc}
Assume that $K$ is a compact set in a complex space $X$ such that 
\begin{enumerate}[\rm (a)]
\item there is a strongly plurisubharmonic function $\rho$ in a neighborhood $K$, and 
\item there is a nonnegative plurisubharmonic function $\tau\ge 0$ in a neighborhood of $K$ 
which vanishes precisely on $K$.
\end{enumerate}
Then, $K$ is holomorphically convex.
\end{proposition}

\begin{proof}
For $\epsilon>0$ the function $t\mapsto \frac{1}{\epsilon - t}$ is strongly increasing and 
convex on $\{t<\epsilon\}$. Hence, for every small enough $\epsilon>0$ the function 
$\frac{1}{\epsilon - \tau}$ is a plurisubharmonic exhaustion function on the domain
$\Omega_\epsilon = \{\tau <\epsilon\}$, which is contained in the domain of $\rho$, 
and hence $\rho+\frac{1}{\epsilon - \tau}$ is a strongly 
plurisubharmonic exhaustion function on $\Omega_\epsilon$. 
Thus, $\Omega_\epsilon$ is a Stein neighborhood of $K$, and 
$K$ is $\Oscr(\Omega_\epsilon)$-convex as shown by the function $\tau$.
\end{proof}

%
%
A compact set $A$ in a complex space $X$ is said to be a {\em complex curve with $\Cscr^r$ 
boundary $bA$} for some $r\in \{1,2,\ldots,\infty\}$ if the following two conditions hold:
\begin{enumerate}[\rm (i)] 
\item $A\setminus bA$ is a closed purely one-dimensional complex analytic subvariety 
of $X\setminus bA$ without compact irreducible components, and
\item every point $p\in bA$ admits an open neighbourhood $U\subset X$
and a holomorphic embedding $\phi : U\hra \Omega\subset \C^N$
such that $\phi(A\cap U)$ is a one dimensional complex submanifold of $\Omega$ with 
$\Cscr^r$ boundary $\phi(bA \cap U)$. In particular, $\phi(bA \cap U)$ is an embedded 
curve of class $\Cscr^r$ in $\C^N$.
\end{enumerate}

The same notion of a complex curve with boundary was used in \cite{DrinovecForstneric2007DMJ}.  
Note that the boundary $bA$ consists of finitely many closed Jordan curves 
of class $\Cscr^r$, and $A$ has at most finitely many singular points in the interior
$\mathring A=A\setminus bA$. The normalization of $A$ is a compact bordered Riemann surface
with $\Cscr^r$ boundary. Every such surface can be realized as a compact domain with real analytic
boundary in a compact Riemann surface by a conformal diffeomorphism which is of 
H\"older class $\Cscr^{r-1,\alpha}$ up to the boundary for any $0<\alpha<1$ 
(see Stout \cite[Theorem 8.1]{Stout1965TAMS} and \cite[Theorem 1.10.10]{AlarconForstnericLopez2021} 
for the discussion and references). Thus, $A$ is a domain with real analytic boundary in a compact 
one-dimensional complex space with finitely many singular points, none of which lie on $bA$.

A compact set $L$ in such a complex curve $A$ is called $\Oscr(A)$-convex if for every point 
$p\in A\setminus L$ there is a function $f\in\Ascr(A)$ (i.e., continuous on $A$ and holomorphic on 
$\mathring A=A\setminus bA$) such that $|f(p)|>\max_L |f|$. 
Equivalently, we can use holomorphic functions on a neighbourhood of $A$ in an ambient one-dimensional 
complex space, or continuous subharmonic functions on $A$. 
By classical function theory, a compact set $L$ in $A$ is 
$\Oscr(A)$-convex if and only if every connected component of $A\setminus L$ intersects $bA$. 
This is an analogue of the Runge condition in open Riemann surfaces, where a compact set 
is holomorphically convex if and only if its complement has no relatively compact connected components. 

We shall be concerned with sets of the form $A\cup K$ having the properties in Theorem \ref{th:AK}.
It will be convenient to have a name for them.

%
%
\begin{definition}\label{def:admissible}
A pair $(A,K)$ of compact sets in a complex space $X$ is {\em admissible} if
\begin{enumerate}[\rm (i)]
\item $A$ is a compact complex curve in $X$ with embedded $\Cscr^2$ boundary 
having no irreducible components without boundary,
\item $K$ is a compact holomorphically convex set (see Definition \ref{def:holoconvex}), and
\item the set $A\cap K$ is $\Oscr(A)$-convex.
\end{enumerate}
An admissible pair $(A,K)$ is {\em special} if in addition $A$ is nonsingular
(a union of finitely many pairwise disjoint embedded complex curves with $\Cscr^2$ boundaries)
and $K$ is a strongly pseudoconvex domain with smooth boundary $bK$ intersecting both 
$A$ and $bA$ transversely. 

If $(A,K)$ is an admissible pair then $A\cup K$ is called an {\em admissible set}.
\end{definition}

Smoothness of $bK$ is meant as explained above, working in an ambient space
via local holomorphic embeddings. (Alternatively, we can use a global holomorphic embedding of 
a Stein neighbourhood of $K$ into $\C^N$.) 
Note that if $(A,K)$ is a special admissible pair, then $A\cap bK$ is a union of finitely many 
pairwise disjoint $\Cscr^2$ arcs connecting pairs of points in $bA$ and intersecting $bA$
transversely at the endpoints, and closed curves contained in $A\setminus bA$.
In fact, these are the relevant properties that will be used in the proof.

%
%
\begin{lemma}\label{lem:reduction}
Let $(A,K)$ be an admissible pair in a complex space $X$.
\begin{enumerate}[\rm 1.]
\item  There exists a strongly plurisubharmonic function in a neighbourhood of $A\cup K$ in $X$.
\item $A\cup K$ is holomorphically convex if and only if there exists a nonnegative 
plurisubharmonic function $\tau \ge 0$ in an open neighbourhood of $A\cup K$ 
such that $\{\rho=0\}=A\cup K$. 
\end{enumerate}
\end{lemma}

\begin{proof}
The first part is \cite[Lemma 2.4]{DrinovecForstneric2007DMJ}. Although the cited lemma is stated for the special case 
$A\cap bK=\varnothing$, its proof also applies in the present situation. 
(An important step in the proof is given by \cite[Lemma 2.5]{DrinovecForstneric2007DMJ} 
which provides an extension of a $\Cscr^2$ strongly subharmonic function on $A$ to a $\Cscr^2$ strongly 
plurisubharmonic function on a neighbourhood of $A$ in $X$.)
The second part is an immediate consequence of the first part and Proposition \ref{prop:hc}.
\end{proof}

Although we do not know whether $A\cup K$ is holomorphically convex for every  
admissible pair $(A,K)$, we will show that this holds for special admissible pairs
and every admissible pair can be enlarged as little as desired to obtain a special admissible pair. 

Theorem \ref{th:AK} clearly follows from the following couple of results,
the first of which is proved in Section \ref{sec:proof}. 

%
%
\begin{theorem}\label{th:AKS} 
For every special admissible pair $(A,K)$ in a complex spaces $X$, the set 
$A\cup K$ is holomorphically convex. 
\end{theorem}

%
%
\begin{lemma}\label{lem:enlarging}
Given an admissible pair $(A,K)$ in a complex space $X$ and an open neighbourhood $U\subset X$
of $A\cup K$, there exists a special admissible pair $(A',K')$ such that $A\cup K\subset A'\cup K'\subset U$.
\end{lemma}

\begin{proof} 
We begin by adding to $K$ the finitely many singular points of $A$. (Since the boundary $bA$ is 
embedded, all singular points are contained in the relative interior $A\setminus bA$ of $A$.)
The new set, still denoted $K$, is clearly holomorphically convex.

Pick a smooth plurisubharmonic exhaustion function $\rho\ge 0$ on a Stein neighbourhood 
$\Omega\subset U$ of $K$ such that $\{\rho=0\}=K$ and $\rho$ is strongly plurisubharmonic 
on the set $\{\rho>0\}=\Omega\setminus K$. For every $c>0$ the domain $\Omega_c=\{\rho<c\}$ 
is Stein and its closure $\overline\Omega_c$ is holomorphically convex. 
We now show that $\overline \Omega_c\cap A$ is $\Oscr(A)$-convex if $c>0$ is small enough. 
Fix $c_0>0$. The function $u=\rho|_{A\cap \Omega}: A\cap \Omega\to \R_+$ 
is plurisubharmonic and vanishes precisely on $A\cap K$. Since this set is $\Oscr(A)$-convex, 
there is a plurisubharmonic function $v:A\to\R$ such that $v<0$ on $A\cap K$ and $v>0$ on 
$A\setminus \Omega_{c_0}$. Let 
\[
	c_1=\min\{v(x): x\in A\setminus \Omega_{c_0}\}>0. 
\]
If $c_2>0$ is chosen such that 
\[
	c_1c_2>\max\{u(x): x\in A\cap b\Omega_{c_0}\}, 
\]
then the function $\phi:A\to\R_+$ given by 
\[
	\phi(x)=\begin{cases} \max\{u(x),c_2 v(x)\}, & \text{if $x\in A\cap \Omega_{c_0}$}; \\
					 c_2v(x), & \text{if $x\in A\setminus  \Omega_{c_0}$}
		  \end{cases}			 
\]
is a nonnegative plurisubharmonic function on $A$ that agrees with $u$ on $\{v=0\}$,  
which is a neighbourhood of $A\cap K$ in $A$. Hence, for all sufficiently small $c>0$ we have that
\[
	A\cap \overline\Omega_c = \{x\in A \cap U: u(x)\le c\} = \{x\in A : \phi\le c\},
\]
so this set is $\Oscr(A)$-convex. For most values of $c$, the boundary $b\Omega_c$ is smooth and 
intersects both $A$ and $bA$ transversely. We now choose $K'=\{\rho\le c\} = \overline \Omega_c$ 
for some $c>0$ satisfying the above conditions. Note that all singular points of $A$ are contained 
in the interior of $K'$. Hence, taking $A\setminus \{\rho<b\}$ for some $0<b<c$ close to $c$ 
and rounding off the corners along $bA\cap \{\rho=b\}$ gives a complex curve $A'\subset A$
without singularities and with $\Cscr^2$ boundary. Clearly, the new pair $(A',K')$ satisfies the lemma.
\end{proof}

%
%
\begin{lemma}\label{lem:hereditary}
Assume that $(A,K)$ is an admissible pair and $\Omega\subset X$ is a Stein neighbourhood 
of $A\cup K$ such that $A\cup K$ is $\Oscr(\Omega)$-convex. Then, for any compact 
$\Oscr(A)$-convex subset $L\subset A$ such that $A\cap K\subset L$ the set
$K\cup L$ is also $\Oscr(\Omega)$-convex.
\end{lemma}

\begin{proof}
We use an idea from Rosay's proof of Rossi's local maximum principle in holomorphic hulls (see \cite[Proposition, p.\ 430]{Rosay2006}), along with the fact that the plurisubharmonic hull of a compact set in a Stein space equals its holomorphically
convex hull \cite{FornaessNarasimhan1980}.

It clearly suffices to show that no point $p\in A\cup K \setminus (K\cup L)=A\setminus L$ belongs to 
$\wh{K\cup L}_{\Oscr(\Omega)}$. Since $L$ is $\Oscr(A)$-convex, there exists
a smooth strongly plurisubharmonic function $f$ on $A$ such that $f(p)>0$ and $f<-1$ on $L$ 
(see \cite[Theorem 5.1.6]{Hormander1990}). By \cite[Lemma 2.5]{DrinovecForstneric2007DMJ} we may assume 
that $f$ is plurisubharmonic on a neighbourhood $U\subset X$ of $A$. Let $V$ denote the 
intersection of the set $\{f\ge -1\}$ with a small compact neighbourhood $W\subset U$ of $A$. Since $f<-1$ on $L$, 
we have that $bV\cap (A\cup K) \subset A\setminus L$ provided $W$ is chosen small enough. 
The function $\tilde \rho=\max\{f,0\}\ge 0$ is then well-defined and plurisubharmonic on a neighbourhood 
of $A\cup K$, it vanishes on a neighbourhood of $K\cup L$, and $\tilde\rho(p)>0$. 
Since $A\cup K$ is $\Oscr(\Omega)$-convex, $\tilde \rho$ extends to a plurisubharmonic function 
on $\Omega$ without changing its values near $A\cup K$. It follows that $p$ does not belong to the 
plurisubharmonic hull of $K\cup L$ in $\Omega$, hence neither to the $\Oscr(\Omega)$-convex hull of $K\cup L$.
\end{proof}


We shall need the following version of a theorem of Stolzenberg \cite{Stolzenberg1966AM}.

%
%
\begin{theorem}\label{th:Stolzenberg}
Assume that $X$ is a Stein space, $K$ is a compact $\Oscr(X)$-convex set in $X$, and
$C_1,\ldots,C_m \subset X$ are arcs or closed Jordan curves of class 
$\Cscr^r$ for some $r\ge 1$. Set $C=\bigcup_{j=1}^m C_j$ and $L=C\cup K$. Then, the set 
\[
	A := \wh{L}_{\Oscr(X)} \setminus L
\] 
is either empty or else a purely one-dimensional complex subvariety of $X\setminus  L$.

Furthermore, if $\overline A$
contains a point $p\in C'_j:=C_j \setminus (\bigcup_{i\ne j}C_i \cup K)$, then it contains the connected 
component $\varGamma$ of $C'_j$ containing $p$, and 
the pair $(A,\varGamma)$ is a local $\Cscr^r$ manifold with boundary at every point in an open 
subset of $\varGamma$ whose complement has zero length.

In particular, if $C_1,\ldots,C_m$ are pairwise disjoint compact embedded $\Cscr^1$ arcs in $X$ 
such that at most one endpoint of each $C_j$ is contained in $K$ and the rest of the arc lies in $X\setminus K$, 
then the set $\bigcup_{j=1}^m C_j \cup K$ is $\Oscr(X)$-convex.
\end{theorem}

Theorem \ref{th:Stolzenberg} is obtained by combining Stolzenberg's theorem in 
\cite{Stolzenberg1966AM}, the boundary regularity theorem for analytic subvarieties given 
by \cite[Theorem, p.\ 255]{Chirka1989} (this result is local, so it also holds in complex spaces), 
and the embedding theorem for Stein spaces into Euclidean spaces; 
see \cite[Theorem 2.4.1]{Forstneric2017E} and the references therein. 

Theorem \ref{th:Stolzenberg} also holds if the curves $C_1,\ldots,C_m$ are piecewise $\Cscr^r$,
since we can add the finitely many nonsmooth points of $C=\bigcup_{j=1}^m C_j$ to the set $K$.
See also the exposition given by Stout in \cite[Theorem 3.1.1]{Stout2007} where 
Stolzeberg's theorem is proved under a weaker regularity assumption on the curves; this generalization
will not be needed.

Another relevant result in this context is the boundary uniqueness theorem for analytic subvarieties
(see Chirka \cite[Proposition 1, p.\ 258]{Chirka1989}); we state it here for future reference.
This result is local, so it also holds in any complex space.

%
%
\begin{proposition}\label{prop:uniqueness} 
Let $M$ be a connected $(2p-1)$-dimensional submanifold of class $\Cscr^1$ in a complex
space $X$, and let $A_1,A_2\subset X\setminus M$ be irreducible $p$-dimensional 
complex subvarieties of $X\setminus M$ whose closures contain $M$. Then, either $A_1=A_2$ 
or else $A_1\cup A_2\cup M$ is a complex analytic subvariety of $X$. 
\end{proposition}
 
Proposition \ref{prop:uniqueness}  implies the following. 
Assume that $K$ is a compact $\Oscr(X)$-convex set in a Stein space $X$ and 
$A$ is a compact complex curve with $\Cscr^1$ boundary in $X$. Then, 
$\wh{A\cup K}_{\Oscr(X)}= A' \cup A\cup K$, where $A'$ is a complex curve 
such that $\overline {A'}\setminus A'\subset K\cup bA$. In particular, 
$A'$ has at most finitely many irreducible components. 

%
%
\begin{lemma}\label{lem:hullA}
If $A$ is a compact complex curve with $\Cscr^2$ boundary in a Stein space $X$,
then $A':=\wh A_{\Oscr(X)}\setminus A$ is either empty or a union of finitely many irreducible complex curves 
$A'_1,\ldots, A'_k$ such that $\overline{A'_j}\setminus A'_j$ is a union of curves in $bA$ for every 
$j=1,\ldots,k$. Furthermore, $A$ admits a Stein neighbourhood $V\subset X$ such that $\wh A_{\Oscr(V)} = A$.
For such $V$ and for any compact $\Oscr(A)$-convex set $K\subset A$ with piecewise $\Cscr^1$ 
boundary we have that $\wh K_{\Oscr(V)} = K$.
\end{lemma}

\begin{proof}
The first claim follows from Theorem \ref{th:Stolzenberg} and Proposition \ref{prop:uniqueness}.
By \cite[Theorem 2.1]{DrinovecForstneric2007DMJ}, $A$ is a Stein compact in $X$. Let us now consider
the second part. Choosing a Stein neighbourhood $V\subset X$ of $A$ which does not
contain any of the finitely many complex curves $A'_1,\ldots, A'_k$ constituting $A':=\wh A_{\Oscr(X)}\setminus A$,
we get that $\wh A_{\Oscr(V)} = A$. Let $K\subset A$ be a compact $\Oscr(A)$-convex set with piecewise $\Cscr^1$ 
boundary $bK=\bigcup_{i=1}^m \varGamma_i$, where $\varGamma_i$ are $\Cscr^1$ arcs
such that any two of them are either disjoint or they meet at an endpoint.
By Theorem \ref{th:Stolzenberg}, the $\Oscr(V)$-hull of 
$bK$ consists of finitely many complex curves $B_1,\ldots, B_l$ with boundaries in $bK$. 
Since $\wh A_{\Oscr(V)} = A$, any curve $B_i$ from this family which is not contained in $K$ 
is a domain in $A\setminus K$ having boundary in $bK$; but there are no such 
domains since $K$ is $\Oscr(A)$-convex. Hence, $\wh K_{\Oscr(V)} = K$. 
\end{proof}

%
%
%
%
\section{Proof of Theorem \ref{th:AKS}}\label{sec:proof}

In this section we prove Theorem \ref{th:AKS}, i.e., we show that 
for every special admissible pair $(A,K)$ in a complex spaces (see Definition \ref{def:admissible}) 
the set $A\cup K$ is holomorphically convex. As explained in the previous section,
this and Lemma \ref{lem:enlarging} imply Theorem \ref{th:AK}.


The proof is based on a couple of lemmas. The first one, Lemma \ref{lem:hc1}, shows that the union 
of $K$ and a collar in $A$ around $A\cap bK$ is holomorphically convex.
The second one, Lemma \ref{lem:hc2}, shows that the union of two special 
holomorphically convex admissible sets, whose intersection is a union of pairwise
disjoint discs, is again holomorphically convex. The proof amounts to
welding plurisubharmonic functions defining the respective admissible sets along their intersection.
Theorem \ref{th:AKS} follows easily from these two lemmas.

Recall that $A$ is a union of finitely many pairwise disjoint 
smooth compact complex curves with $\Cscr^2$ boundaries, $K=\{\rho\le 0\}$ where 
$\rho$ is a strongly plurisubharmonic exhaustion function on a Stein neighbourhood $\Omega$ of $K$
(hence $K$ is $\Oscr(\Omega)$-convex), and $0$ is a regular value of both $\rho|_{A}$ and $\rho|_{bA}$.
By compactness of $A$ there is a $\delta>0$ such that 
\begin{equation}\label{eq:regularvalue}
	\text{every number $c\in [0,\delta]$ is a regular value of $\rho|_{bA}$ and $\rho|_{bA}$.} 
\end{equation}
For any $c\in [0,\delta]$ we set 
\begin{equation}\label{eq:Gammac}
	A_c = \{x\in A\cap \Omega: 0\le \rho(x)\le c\},\quad  \ 
	\varGamma_c = \{x\in A\cap \Omega: \rho(x)= c\}. 
\end{equation}
Note that $A_c$ is an exterior collar in $A$ around $A\cap bK$.
Condition \eqref{eq:regularvalue} implies that $\varGamma_c$  
is a union of finitely many pairwise disjoints arcs and closed curves of class $\Cscr^2$,
and there is a diffeomorphism $\varGamma_0\times [0,c] \to A_c$ mapping 
$\varGamma_0\times \{c'\}$ onto $\varGamma_{c'}$ for every $c'\in [0,c]$.
Note that $(K,A_c)$ is a special admissible pair, except that $bA_c$ is only piecewise $\Cscr^2$. 

%
%
\begin{lemma}\label{lem:hc1}
For every $c\in [0,\delta]$ the set $K\cup A_c$ is $\Oscr(\Omega)$-convex.
If $\varGamma_c$ does not contain any closed curves then $K\cup \varGamma_c$ 
is also $\Oscr(\Omega)$-convex.
\end{lemma}

\begin{proof}
Conditions \eqref{eq:regularvalue} imply that the set $A_c$ in \eqref{eq:Gammac} 
is an embedded compact complex curve with piecewise $\Cscr^2$ boundary (and without singularities)
which is a disjoint union of at most finitely many closed discs $D_1,\ldots,D_m$ and annuli $E_1,\ldots, E_l$. 
Theorem \ref{th:Stolzenberg} applied to the set $K\cup bA_c$ shows that its $\Oscr(\Omega)$-convex 
hull contains $K\cup A_c$ together with possibly other closed, irreducible, relatively
compact complex curves in $\Omega\setminus (K\cup bA_c)$. 
Suppose that $C$ is such a curve which is not contained in $A_c$. 
If the closure $\overline C$ intersects one of the boundaries $bD_j\setminus K$, then
it contains this entire set, and by Proposition \ref{prop:uniqueness}
the union $C\cup D_j$ is a complex curve along $bD_j\setminus K$. A similar argument
holds if $\overline C$ intersects one of the closed curves $bE_j\setminus K$. 
The upshot is that $C$ together with some connected components of the set $A_c\setminus K$ 
defines a closed complex curve in $\Omega\setminus K$ whose boundary lies in $K$. 
This is impossible since $K$ is $\Oscr(\Omega)$-convex. Hence, $K\cup A_c$ is
$\Oscr(\Omega)$-convex. 

If $\varGamma_c$ does not contain any closed curves, then $K\cup \varGamma_c$
is $\Oscr(\Omega)$-convex by the last part of Theorem \ref{th:Stolzenberg}. 
\end{proof}

The next lemma explains how to glue a pair of holomorphically convex special admissible sets under 
a suitable geometric condition on their intersection. This is the main new idea in our approach.

%
%
\begin{lemma}\label{lem:hc2}
Assume that $(A_0,K_0)$ and $(A_1,K_1)$ are special admissible pairs in a complex space $X$ 
satisfying the following conditions.
\begin{enumerate}[\rm (a)]
\item The special admissible set $L_i:=A_i\cup K_i$ for $i=0,1$ is holomorphically convex.
\item $\overline {L_0\setminus L_1} \cap \overline {L_1\setminus L_0}=\varnothing$.
\item $L_0\cap L_1=A_0\cap A_1$ is a union of pairwise disjoint discs $D_1,\ldots,D_m$
which do not intersect the set $K:=K_0\cup K_1$. 
\end{enumerate}
Then, the pair $(A=A_0\cup A_1,K=K_0\cup K_1)$ is a special admissible pair and the set
$L_0\cup L_1= A\cup K$ is holomorphically convex.
\end{lemma}


\begin{proof}
It is clear that $(A,K)$ is a special admissible pair. In particular, the conditions imply that 
any point in $A\setminus K$ can be connected by a path in this set to a point in $bA\setminus K$
and the other conditions trivially follow from the hypotheses.

By the assumption, the admissible set $L_i$ is $\Oscr(\Omega_i)$-convex
in a Stein domain $\Omega_i\subset X$. 

The boundary of every disc $D_j$ in $L_0\cap L_1$ is of the form
\begin{equation}\label{eq:bDj}
	bD_j=\alpha_{0,j} \cup \beta_{j} \cup \alpha_{1,j} \cup \gamma_j,
\end{equation}
where the arc $\alpha_{0,j}\subset bA_0$ is the closure of $bD_j\cap (A_1\setminus bA_1)$,
the arc $\alpha_{1,j}\subset bA_1$ is the closure of $bD_j\cap (A_0\setminus bA_0)$,
and the arcs $\beta_j,\gamma_j\subset bA_0\cap bA_1$ connect the respective endpoints of 
$\alpha_{0,j}$ and $\alpha_{1,j}$. For $i=0,1$ we define the following sets:
\[
	\varGamma_i =\bigcup_{j=1}^m \alpha_{i,j},\qquad 
	L'_0 = \overline {L_0\setminus L_1}, \qquad 
	L'_1 = \overline {L_1\setminus L_0}.
\]	
Note that the set $\overline{A_0\setminus A_1} \cup \varGamma_0\subset A_0$ contains
$A_0\cap K_0$ and is clearly $\Oscr(A_0)$-convex. 
Hence, Lemma \ref{lem:hereditary} implies that the set 
\[
	L'_0\cup \varGamma_0=K_0\cup \overline{A_0\setminus A_1} \cup \varGamma_0
\]
is $\Oscr(\Omega_0)$-convex. The symmetric argument shows that $L'_1\cup \varGamma_1$
is $\Oscr(\Omega_1)$-convex.

Since $L'_i$ and $\varGamma_i$ are disjoint for $i=0,1$, it follows from general properties of
$\Oscr(\Omega_i)$-convex sets that each arc $\alpha_{i,j}$ $(j=1,\ldots,m)$ 
has a small compact tubular neighbourhood $\Lambda_{i,j}\subset \Omega_i$ 
(defined as a sublevel set of a nonnegative strongly plurisubharmonic function 
vanishing precisely on $\alpha_{i,j}$; see \cite[Corollary 3.5.2]{Forstneric2017E}) 
such that these sets are pairwise disjoint, they are disjoint from $L'_i$
and such that, setting $\Lambda_i=\bigcup_{j=1}^m \Lambda_{i,j}$, 
the compact set $L'_i\cup \Lambda_i$ is $\Oscr(\Omega_i)$-convex for $i=0,1$. 

Consider the union of the sets $L'_i\cup \Lambda_i$ and $B=\bigcup_{j=1}^m \beta_j \cup \gamma_j$,
the latter consisting of pairwise disjoint arcs connecting $L'_i$ to $\Lambda_{i}$.
By Theorem \ref{th:Stolzenberg}, its $\Oscr(\Omega_i)$-convex hull consists of $L_i \cup \Lambda_i$,
which is obtained by filling in the discs $\bigcup_{j=1}^m D_j$, and perhaps some other 
irreducible complex curves in $\Omega_i\setminus(L'_i\cup \Lambda_i \cup B)$
having their ends in $L'_i\cup \Lambda_i \cup B$.
However, the same argument as in the proof of Lemma \ref{lem:hc1} shows that any such curve 
$C$ which is not one of the discs $D_j$ creates a complex curve in the Stein domain
$\Omega_i$ with boundary in the $\Oscr(\Omega_i)$-convex set $L'_i\cup \Lambda_i$, a contradiction.
More precisely, if the closure of $C$ contains a point $p\in \beta_j$, then it contains this entire arc
and $C \cup D_j$ is analytic along $\beta_j$. Since every end of $C$ is connected,
our geometric situation implies that the end containing $\beta_j$ also contains the other arc 
$\gamma_j$ in the boundary of the disc $D_j$,  so $C\cup D_j$ is a complex curve
along these two arcs. The curve $C$ may have other ends of the same type. 
The upshot is that $C$ together with some of the discs $D_j$ closes up to 
a complex curve $\wt C$ having its ends in $L'_i\cup \Lambda_i$. 
Since this set is $\Oscr(\Omega_i)$-convex, we arrived at a contradiction. 

This contradiction shows that the compact set 
$L_i \cup \Lambda_i$ is $\Oscr(\Omega_i)$-convex for $i=0,1$.
Hence, there is a plurisubharmonic function $\tau_i\ge 0$ on $\Omega_i$ 
vanishing precisely on $L_i \cup \Lambda_i$. We claim that the function $\tau=\max\{\tau_0,\tau_1\}\ge 0$ 
is then well-defined (and hence plurisubharmonic) on a neighbourhood of $A\cup K=L_0\cup L_1$, 
and it vanishes precisely on $A\cup K$.
To see this, note that $\tau=\tau_1$ on $\Lambda_0$ since $\tau_0$ vanishes there, 
so the function $\tau_1$ takes over in the maximum before we run out of 
the domain of $\tau_0$, provided that we stay close enough to $A\cup K$. By taking 
$\tau=\tau_1$ in a neighbourhood of $\overline {L_1\setminus L_0}$ we see that 
$\tau$ is well-defined in a neighbourhood of $L_1$.
The symmetric argument shows that $\tau$ is well-defined in a neighbourhood of $L_0$ if we take 
$\tau=\tau_0$ near $\overline {L_0\setminus L_1}$.
Finally, any point $p\in L_0\cap L_1$ has a neighbourhood $U\subset X$ such that
both functions $\tau_0,\tau_1$ are defined on $U$ and at least one of them is positive
on $U\setminus (L_0\cup L_1)=U\setminus (A\cup K)$. (The last condition holds because 
the sets $\Lambda_0$ and $\Lambda_1$ are disjoint.) This proves the claim.

It follows from part 2 of Lemma \ref{lem:reduction} that $L_0\cup L_1$ is holomorphically convex.
\end{proof}

\begin{remark}\label{rem:hc2}
The conclusion of Lemma \ref{lem:hc2} is false in general if the set $L_0\cap L_1=A_0\cap A_1$ contains 
a complex annulus $E$. Indeed, such an annulus may be contained in a bigger annulus 
$E'$ in $A=A_0\cup A_1$ with boundary in $K= K_0\cup K_1$, and hence it may generate 
a nontrivial envelope of holomorphy of $A\cup K$. Note also that in such case the set 
$L_0\cup L_1$ is not admissible since $E'$ lies in the $\Oscr(A)$-convex hull of $A\cap K$.
\end{remark}

%
%
\begin{proof}[Proof of Theorem \ref{th:AKS}]
This follows by first using Lemma \ref{lem:hc1} and then performing a finite inductive application of
Lemma \ref{lem:hc2}. We begin with the special admissible set 
$A_c\cup K$ \eqref{eq:Gammac}, which is holomorphically convex by Lemma \ref{lem:hc1}.
We then successively attach discs with $\Cscr^2$ boundaries in $A\setminus K$ 
such that each attachment satisfies the conditions of Lemma \ref{lem:hc2}, 
reaching $A\cup K$ in finitely many steps. 
Every such disc is holomorphically convex by Lemma \ref{lem:hullA}. 
(In fact, we can choose small discs which may be presented as graphs in local holomorphic coordinates,
so holomorphic convexity is elementary to establish.) Every disc can be chosen such that its 
intersection with the previous set is either empty or it consists of one or two disc. 
The first case is used to add a new connected component, and the last one (two discs)
to change the topology of the set, 
i.e., to increase the genus or reduce the number of connected components. 
The hypothesis that $A\cap K$ is $\Oscr(A)$-convex guarantees that there is 
no need to attach a disc to the previous set along its entire boundary curve. 
A precise geometric description of this procedure can be found in standard sources
on Riemann surfaces, and also in \cite[Section 1.4]{AlarconForstnericLopez2021}.
\end{proof}

%
%
\subsection*{Acknowledgements}
Research was supported in part by the research program P1-0291 and grants J1-9104 and J1-3005 
from ARRS, Republic of Slovenia.



\begin{thebibliography}{10}

\bibitem{AlarconForstnericLopez2021}
A.~Alarc\'{o}n, F.~Forstneri\v{c}, and F.~J. L\'{o}pez.
\newblock {\em Minimal surfaces from a complex analytic viewpoint}.
\newblock Springer Monographs in Mathematics. Springer, Cham, 2021.

\bibitem{Chirka1989}
E.~M. Chirka.
\newblock {\em Complex analytic sets}, volume~46 of {\em Mathematics and its
  Applications (Soviet Series)}.
\newblock Kluwer Academic Publishers Group, Dordrecht, 1989.
\newblock Translated from the Russian by R. A. M. Hoksbergen.

\bibitem{Coltoiu1990}
M.~Col{\c{t}}oiu.
\newblock Complete locally pluripolar sets.
\newblock {\em J. Reine Angew. Math.}, 412:108--112, 1990.

\bibitem{Demailly1990}
J.-P. Demailly.
\newblock Cohomology of {$q$}-convex spaces in top degrees.
\newblock {\em Math. Z.}, 204(2):283--295, 1990.

\bibitem{DrinovecForstneric2007DMJ}
B.~Drinovec~Drnov{\v{s}}ek and F.~Forstneri\v{c}.
\newblock Holomorphic curves in complex spaces.
\newblock {\em Duke Math. J.}, 139(2):203--253, 2007.

\bibitem{FornaessNarasimhan1980}
J.~E. Forn{\ae}ss and R.~Narasimhan.
\newblock The {L}evi problem on complex spaces with singularities.
\newblock {\em Math. Ann.}, 248(1):47--72, 1980.

\bibitem{Forstneric2005AIF}
F.~Forstneri\v{c}.
\newblock Extending holomorphic mappings from subvarieties in {S}tein
  manifolds.
\newblock {\em Ann. Inst. Fourier (Grenoble)}, 55(3):733--751, 2005.

\bibitem{Forstneric2017E}
F.~Forstneri\v{c}.
\newblock {\em Stein manifolds and holomorphic mappings (The homotopy principle
  in complex analysis)}, volume~56 of {\em Ergebnisse der Mathematik und ihrer
  Grenzgebiete. 3. Folge}.
\newblock Springer, Cham, second edition, 2017.

\bibitem{Forstneric2022JMAA}
F.~Forstneri\v{c}.
\newblock Euclidean domains in complex manifolds.
\newblock {\em J. Math. Anal. Appl.}, 506(1):Paper No. 125660, 17, 2022.

\bibitem{Grauert1958AM}
H.~Grauert.
\newblock On {L}evi's problem and the imbedding of real-analytic manifolds.
\newblock {\em Ann. of Math. (2)}, 68:460--472, 1958.

\bibitem{GrauertPeternellRemmert1994}
H.~Grauert, T.~Peternell, and R.~Remmert, editors.
\newblock {\em Several complex variables. {VII}}, volume~74 of {\em
  Encyclopaedia of Mathematical Sciences}.
\newblock Springer-Verlag, Berlin, 1994.
\newblock Sheaf-theoretical methods in complex analysis, A reprint of {{\i}t
  Current problems in mathematics. Fundamental directions. Vol. 74} (Russian),
  Vseross. Inst. Nauchn. i Tekhn. Inform. (VINITI), Moscow.

\bibitem{GrauertRemmert1958}
H.~Grauert and R.~Remmert.
\newblock Komplexe {R}\"aume.
\newblock {\em Math. Ann.}, 136:245--318, 1958.

\bibitem{GrauertRemmert1979}
H.~Grauert and R.~Remmert.
\newblock {\em Theory of {S}tein spaces}, volume 236 of {\em Grundlehren Math.
  Wiss.}
\newblock Springer-Verlag, Berlin-New York, 1979.
\newblock Translated from the German by Alan Huckleberry.

\bibitem{Hormander1990}
L.~H{\"o}rmander.
\newblock {\em An introduction to complex analysis in several variables},
  volume~7 of {\em North-Holland Mathematical Library}.
\newblock North-Holland Publishing Co., Amsterdam, third edition, 1990.

\bibitem{Mihalache1996}
N.~Mihalache.
\newblock Voisinages de {S}tein pour les surfaces de {R}iemann avec bord
  immerg\'{e}es dans l'espace projectif.
\newblock {\em Bull. Sci. Math.}, 120(4):397--404, 1996.

\bibitem{Narasimhan1961}
R.~Narasimhan.
\newblock The {L}evi problem for complex spaces.
\newblock {\em Math. Ann.}, 142:355--365, 1960/1961.

\bibitem{Poletsky2013}
E.~A. Poletsky.
\newblock Stein neighborhoods of graphs of holomorphic mappings.
\newblock {\em J. Reine Angew. Math.}, 684:187--198, 2013.

\bibitem{Rosay2006}
J.-P. Rosay.
\newblock Polynomial convexity and {R}ossi's local maximum principle.
\newblock {\em Michigan Math. J.}, 54(2):427--438, 2006.

\bibitem{Siu1976}
Y.~T. Siu.
\newblock Every {S}tein subvariety admits a {S}tein neighborhood.
\newblock {\em Invent. Math.}, 38(1):89--100, 1976/77.

\bibitem{Starcic2008}
T.~Star{\v c}i{\v c}.
\newblock Stein neighborhood bases of embedded strongly pseudoconvex domains
  and approximation of mappings.
\newblock {\em J. Geom. Anal.}, 18(4):1133--1158, 2008.

\bibitem{Stolzenberg1966AM}
G.~Stolzenberg.
\newblock Uniform approximation on smooth curves.
\newblock {\em Acta Math.}, 115:185--198, 1966.

\bibitem{Stout1965TAMS}
E.~L. Stout.
\newblock Bounded holomorphic functions on finite {R}iemann surfaces.
\newblock {\em Trans. Amer. Math. Soc.}, 120:255--285, 1965.

\bibitem{Stout2007}
E.~L. Stout.
\newblock {\em Polynomial convexity}, volume 261 of {\em Progress in
  Mathematics}.
\newblock Birkh\"auser Boston, Inc., Boston, MA, 2007.

\bibitem{Suzuki1978}
M.~Suzuki.
\newblock Sur les int\'{e}grales premi\`eres de certains feuilletages
  analytiques complexes.
\newblock In {\em Fonctions de plusieurs variables complexes, {III} ({S}\'{e}m.
  {F}ran\c{c}ois {N}orguet, 1975--1977)}, volume 670 of {\em Lecture Notes in
  Math.}, pages 53--79, 394. Springer, Berlin, 1978.

\end{thebibliography}


\vspace*{5mm}
\noindent Franc Forstneri\v c

\noindent Faculty of Mathematics and Physics, University of Ljubljana, Jadranska 19, SI--1000 Ljubljana, Slovenia, and 

\noindent 
Institute of Mathematics, Physics and Mechanics, Jadranska 19, SI--1000 Ljubljana, Slovenia

\noindent e-mail: {\tt franc.forstneric@fmf.uni-lj.si}

\end{document}